\documentclass[graybox]{svmult}

\usepackage{amsfonts,amsmath,amstext,latexsym,amssymb,physics}
\usepackage{mathptmx}       
\usepackage{helvet}         
\usepackage{courier}        
\usepackage{type1cm}        
\usepackage{makeidx}         
\usepackage{graphicx}        
\usepackage{multicol}        
\usepackage[bottom]{footmisc}

\makeindex             

\newtheorem{exam}{Example}[section]

\newtheorem{lem}{Lemma}[section]

\begin{document}

\title*{Centralizers in PBW Extensions}
  \titlerunning{Centralizers in PBW extensions} 
\author{Alex Behakanira Tumwesigye, Johan Richter, Sergei Silvestrov }
 \authorrunning{A. B. Tumwesigye, J. Richter \& S. Silvestrov}
\institute{Alex Behakanira Tumwesigye \at Department of Mathematics, College of Natural Sciences, Makerere University, Box 7062, Kampala, Uganda. \\ \email{alexbt@cns.mak.ac.ug}
\and Johan Richter
\at Department of Mathematics and Natural Sciences, Blekinge Institute of Technology, SE-37179 Karlskrona, Sweden. \\ \email{johan.richter@bth.se}
\and Sergei Silvestrov
\at Division of Applied Mathematics, School of Education, Culture and Communication, M\"alardalen University, Box 883, 72123 V\"aster{\aa}s, Sweden. \\ \email{sergei.silvestrov@mdh.se}
}
%
%
\maketitle

\abstract*{In this article we give a description for the centralizer of the coefficient ring $R$ in the skew PBW extension $\sigma(R)<x_1,x_2,\cdots,x_n>.$ We give an explicit description in the quasi-commutative case and state a necessary condition in the general case. We also consider the PBW extension $\sigma(\mathcal{A})<x_1,x_2,\cdots,x_n>$ of the algebra of functions with finite support on a countable set, describing the centralizer of $\mathcal{A}$ and the center of the skew PBW extension.}

\abstract{In this article we give a description for the centralizer of the coefficient ring $R$ in the skew PBW extension $\sigma(R)<x_1,x_2,\cdots,x_n>.$ We give an explicit description in the quasi-commutative case and state a necessary condition in the general case. We also consider the PBW extension $\sigma(\mathcal{A})<x_1,x_2,\cdots,x_n>$ of the algebra of functions with finite support on a countable set, describing the centralizer of $\mathcal{A}$ and the center of the skew PBW extension.}\vspace{0.2cm}


MSC Classification: $16W55$
\keywords{Ring extensions, skew PBW extensions, contralizer, center}


\section{Introduction}
Skew PBW (Poincare-Birkoff-Witt) extensions also known as $\sigma-$PBW extensions are a wide class of non commutative rings which were introduced in \cite{GallegoPBWRST}. Skew PBW extensions include many rings and algebras arising in quantum mechanics such as the classical PBW extensions, Weyl algebras, enveloping algebras of finite dimensional Lie algebras, iterated Ore extensions of injective type and many others. See for example \cite{Artamonov1PBWRST,Artamonov2PBWRST,BellPBWRST,GallegoPBWRST,Lezama1PBWRST,Lezama2PBWRST} for examples of rings and algebras which are skew PBW and some ring theory properties that have been investigated.
 \par 
 In this article we describe the centralizer of the coefficient ring $R$ in the skew PBW extension $\sigma(R)<x_1,x_2,\cdots,x_n>$. Specifically, we extend some of the results in \cite{Richter1PBWRST} and \cite{Richter2PBWRST} in the setting of Ore extensions, to the more general setting of skew $PBW$ extensions. We also describe the center of the skew PBW extension $\tilde{\tau}\qty(\mathbb{R}^{\Omega})<x_1,x_2,\cdots,x_n>$ where $\mathbb{R}^{\Omega}$ is the algebra of real valued functions on a finite set $\Omega.$  Centers of many algebras that can be interpreted as skew PBW extensions have been described in \cite{Lezama3PBWRST}, but that is in a different setting to the one here. The paper is arranged as follows.
 \par 
 In section \ref{DefsPBWRST} we state definitions and preliminaries of skew PBW extensions. Most of the work in this section is based on \cite{GallegoPBWRST}. In section \ref{CentPBWRST}, we give a description of the centralizer of the coefficient ring $R$ in the skew PBW extension for an integral domain $R.$ We give a full description of the centralizer in the quasi-commutative case and state a necessary condition in the general case. In section \ref{FunAlgPBWRST}, we turn attention to the skew PBW extension for the algebra of  real-valued functions $\mathbb{R}^{\Omega}$ on a finite set $\Omega$. We prove that this algebra is isomorphic to the algebra $\mathcal{A}$ of piecewise constant functions on the real line with a finite number of jumps. We then give a full description of the centralizer of the coefficient algebra $\mathbb{R}^{\Omega}$ and the center of the PBW extension in the quasi-commutative case, and state a necessary condition in the general case. We finish the section by describing the centralizer of $\mathcal{A}$ in the skew PBW extension $\tilde{\sigma}(\mathcal{A})<x_1,x_2,\cdots,x_n>$ in terms of $Sep^{\alpha}(\Omega)$ via the isomorphism between $\mathcal{A}$ and $\mathbb{R}^{\Omega}.$
\section{Definitions and Preliminary Notions}\label{DefsPBWRST}
In this section we define skew PBW extensions and state some  preliminary results concerning skew PBW extensions.
\begin{definition}\label{PBWdefRST}
Let $R$ and $A$ be rings. We say that $A$ is a $\sigma-$PBW extension of $R$ (or skew PBW extension), if the following conditions hold:
\begin{itemize}
\item[(a)] $R\subseteq A.$
\item[(b)] There exist finite elements $x_1,\cdots,x_n$ such that $A$ is a left $R-$free module with basis 
$$
Mon(A):=\qty{x^{\alpha}=x_1^{\alpha_1}\cdots x_n^{\alpha_n}\ :\ \alpha=(\alpha_1,\cdots,\alpha_n)\in \mathbb{N}^n}.
$$
\item[(c)] For every $1\leqslant i\leqslant n$ and $r\in R\setminus \qty{0},$ there exists $c_{i,r}\in R\setminus \qty{0}$ such that
$$
x_ir-c_{i,r}x_i\in R.
$$
\item[(d)] For every $1\leqslant i,j\leqslant n$ there exists $c_{i,j}\in R\setminus \qty{0}$ such that
$$x_jx_i-c_{i,j}x_ix_j\in R+Rx_1+\cdots+Rx_n.$$
\end{itemize}
Under these conditions we write $A=\sigma(R)\langle x_1,\cdots,x_n\rangle.$
\end{definition}

The following result \cite[Proposition 3]{GallegoPBWRST} is crucial in establishing the link between skew PBW extensions and many well known algebras.\\
\begin{proposition}\label{prop1RST}
Let $A$ be a PBW extension of $R.$ Then for every $1\leqslant i\leqslant n,$ there exists an injective ring endomorphism $\sigma_i:R\to R$ and a $\sigma_i-$derivation $\delta_i:R\to R$ such that 
$$x_ir=\sigma_i(r)x+\delta_i(r)$$ for each $r\in R.$
\end{proposition}
A particular case of $\sigma-$PBW extension is when all derivations $\delta_i$ are zero. Another interesting case is when all $\sigma_i$ are bijective. This gives motivation to the next definition.
\begin{definition}\label{quasicomm}
Let $A$ be a $\sigma-$PBW extension.
\begin{enumerate}
\item $A$ is quasi-commutative if the conditions $(c)$ and $(d)$ in Definition \ref{PBWdef} are replaced by:
\begin{enumerate}
\item[(c')]   For every $1\leqslant i\leqslant n$ and $r\in R\setminus \qty{0}$, there exists $c_{i,r}\in R\setminus\qty{0}$ such that 
$$x_ir=c_{i,r}x_i.$$
\item[(d')]   For every $1\leqslant i,j\leqslant n$ there exists $c_{i,j}\in R\setminus \qty{0}$ such that
$$x_jx_i=c_{i,j}x_ix_j.$$
\end{enumerate}
\item $A$ is bijective if $\sigma_i$ is bijective for every $1\leqslant i\leqslant n$ and $c_{i,j}$ is invertible for any $1\leqslant i<j\leqslant n.$
\end{enumerate}
\end{definition}
In the next definition, we state some useful notation.
\begin{definition}
Let $A$ be a $\sigma-$PBW extension of $R$ with endomorphisms $\sigma_i,\ 1\leqslant i\leqslant n,$ as in Proposition \ref{prop1RST}.
\begin{itemize}
\item[(a)] For $\alpha=(\alpha_1,\cdots,\alpha_n)\in \mathbb{N}^n,\ \sigma^{\alpha}:=\sigma_1^{\alpha_1}\cdots \sigma_n^{\alpha_n},\ \abs{\alpha}:=\alpha_1+\cdots+\alpha_n.$ If $\beta=(\beta_1,\cdots,\beta_n)\in \mathbb{N}^n,$ then $\alpha+\beta:=(\alpha_1+\beta_1,\cdots,\alpha_n+\beta_n).$
\item[(b)] For $X=x^{\alpha}\in Mon(A),\ exp(X):=\alpha$ and $deg(X)=\abs{\alpha}.$
\item[(c)] Let $0\neq f\in A$ such that $f=c_1X_1+\cdots+c_tX_t$ with $X_i\in Mon(A)$ and $c_i\in R\setminus \qty{0}$ then $deg(f)=\max\qty{deg(X_i)}_{i=1}^t.$
\end{itemize}
\end{definition}
\section{Centralizers in Skew PBW extensions}\label{CentPBWRST}
In this section we give a description of the centralizer $C(R)$ of the (commutative) coefficient ring $R$ in the skew PBW extension $\sigma(R)<x_1,x_2,\cdots,x_n>.$ We start by giving a full description of the centralizer in the quasi commutative case and then give a necessary condition in the general case.
\begin{theorem}\label{thm1PBWRST}
Let $R$ be a commutative ring and suppose that for all $1\leqslant i\leqslant n,\ \delta_i=0.$ Then the centralizer $C(R)$ of $R$ in the skew $PBW$ extension $\sigma(R)<x_1,\cdots,x_n>$ is given by 
$$
C(R)=\qty{\sum_{\alpha}f_{\alpha}x^{\alpha}\ :\ (\forall \ r\in R),\ \qty(\sigma^{\alpha}(r)-r)f_{\alpha}=0}.
$$
\end{theorem}
\begin{proof}
An element $f=\sum\limits_{\alpha}f_{\alpha}x^{\alpha}\in \sigma(R)<x_1,\cdots,x_n>$ belongs to $C(R)$ if and only if for every $r\in R,\ rf=fr.$ $$rf=r\sum\limits_{\alpha}f_{\alpha}x^{\alpha}=\sum\limits_{\alpha}rf_{\alpha}x^{\alpha}.$$ 
On the other hand, if $\delta_i=0$ for $1\leqslant i\leqslant n,$ then for every $\alpha=(\alpha_1,\cdots,\alpha_n)\in \mathbb{N}^n$ and every $r\in R$ we have; 
$$x^{\alpha}r=\sigma^{\alpha}(r)x^{\alpha}.$$
Therefore 
\begin{align*}
fr&=\qty(\sum\limits_{\alpha}f_{\alpha}x^{\alpha})r\\
&=\sum\limits_{\alpha}f_{\alpha}x^{\alpha}r\\
&=\sum\limits_{\alpha}f_{\alpha}\sigma^{\alpha}(r)x^{\alpha}.
\end{align*}
Since $R$ is commutative, it follows that $rf=fr$ if and only if 
$$\qty(\sigma^{\alpha}(r)-r)f_{\alpha}=0.$$ 
Therefore 
$$
C(R)=\qty{\sum_{\alpha}f_{\alpha}x^{\alpha}\ :\ (\forall \ r\in R),\ \qty(\sigma^{\alpha}(r)-r)f_{\alpha}=0}.
$$
\end{proof}
In the general case, we have the following necessary condition.
\begin{theorem}\label{thm2PBWRST}
Let $R$ be a commutative ring. If an element $\sum_{\alpha}f_{\alpha}x^{\alpha}\in \sigma(R)<x_1,\cdots,x_n>$ belongs to the centralizer $C(R),$ then $\qty(\sigma^{\alpha}(r)-r)f_{\alpha}=0$ for all $\alpha\in \mathbb{N}^n.$
\end{theorem}
\begin{proof}
Suppose an element $f=\sum_{\alpha}f_{\alpha}x^{\alpha}\in \sigma(R)<x_1,\cdots,x_n>$ belongs to the centralizer of $R.$ Then $fr=rf$ for every $r\in R.$ Now,
$$rf=r\sum\limits_{\alpha}f_{\alpha}x^{\alpha}=\sum\limits_{\alpha}rf_{\alpha}x^{\alpha}.$$
On the other hand, by \cite[Theorem 7]{GallegoPBWRST},  for every $x^{\alpha} \in Mon\qty(\sigma(R)<x_1,\cdots,x_n>)$ and every $r\in R$ we have 
$$x^{\alpha}r=\sigma^{\alpha}(r)x^{\alpha}+p_{\alpha,r}$$ where $p_{\alpha,r}\in R[x_1,\cdots,x_n]$ such that $p_{\alpha,r}=0$ or $deg(p_{\alpha,r})<\abs{\alpha}.$ Therefore;
\begin{eqnarray*}
fr &=&\sum_{\alpha}\qty(f_{\alpha}x^{\alpha})r\\
&=&\sum_{\alpha}f_{\alpha}\qty(x^{\alpha}r)\\
&=&\sum_{\alpha}f_{\alpha}\qty(\sigma^{\alpha}(r)x^{\alpha}+p_{\alpha,r}).
\end{eqnarray*}
Comparing the leading coefficients and using the fact that $R$ is commutative, we see that if $fr=rf,$ then
$$
\qty(\sigma^{\alpha}(r)-r)f_{\alpha}=0 \text{ for all }\alpha.
$$
\end{proof}
As a result, we have the following Corollary which is the extension of \cite[Proposition 3.3]{Richter2PBWRST} to the skew PBW extension case.
\begin{corollary}
Let $R$ be a commutative ring. If for every $\alpha\in \mathbb{N}^n$ there exists $r\in R$ such that $\qty(\sigma^{\alpha}(r)-r)$ is a regular element, then $C(R)=R.$
\end{corollary}
\begin{proof}
Suppose  $f=\sum_{\alpha}f_{\alpha}x^{\alpha}\in \sigma(\mathcal{A})<x_1,\cdots,x_n>$ is a non-constant element of degree $\alpha$ which belongs to the centralizer of $R.$ Then $fr=rf$ for every $r\in R.$ Now,
$$rf=r\sum\limits_{\alpha}f_{\alpha}x^{\alpha}=\sum\limits_{\alpha}rf_{\alpha}x^{\alpha}.$$
On the other hand, by \cite[Theorem 7]{GallegoPBWRST},  for every $x^{\alpha} \in Mon(A)$ and every $r\in \mathcal{A}$ we have 
$$x^{\alpha}r=\sigma^{\alpha}(r)+p_{\alpha,r}$$ where $p_{\alpha,r}=0$ or $deg(p_{\alpha,r})<\abs{\alpha}$ if $p_{\alpha,r}\neq 0.$ Therefore;
\begin{eqnarray*}
fr &=&\sum_{\alpha}\qty(f_{\alpha}x^{\alpha})r\\
&=&\sum_{\alpha}f_{\alpha}\qty(x^{\alpha}r)\\
&=&\sum_{\alpha}f_{\alpha}\qty(\sigma^{\alpha}(r)x^{\alpha}+p_{\alpha,r})
\end{eqnarray*}
Equating coefficients and using commutativity of $R$, we get
$$rf_{\alpha}=\sigma^{\alpha}(r)f_{\alpha}, \ \ \text{or equivalently }\qty(\sigma^{\alpha}(r)-r)f_{\alpha}=0.$$
Since $\sigma^{\alpha}(r)-r$ is a regular element, then we  $f_{\alpha}=0$ for all $\alpha,$ which is a contradiction.
\end{proof}
\section{Skew PBW extensions of function algebras}\label{FunAlgPBWRST}
In this section we treat skew PBW extensions for the algebra of functions on a finite set. In \cite{AlexABTPBWRST}, the commutant of the coefficient algebra in the crossed product algebra for the algebra of piecewise constant functions on the real line was  described. However, as we show in Proposition \ref{IsomoPropRST} below, the algebra of piecewise constant functions on the real line is isomorphic to the algebra of real-valued functions on some finite set.\par 
Let $\mathbb{P}=\bigcup\limits_{k=0}^{2N}I_k$ be a partition of $\mathbb{R},$ where $I_{k}=(t_k,t_{k+1}),$ for $k=0,1,\cdots, N$ with $t_0=-\infty$ and $t_{N+1}=\infty$ and $I_{N+k}=\qty{t_k},k=1,\cdots,N$ and let $\mathcal{A}$ be the algebra of functions which are constant on the intervals $I_k,\ k=0,1,\cdots,2N.$ Then $\mathcal{A}$ is the algebra of piecewise constant functions $h:\mathbb{R}\to \mathbb{R}$ with $N$ fixed jumps at points $t_1,\cdots,t_N.$\par 
 Let $\Omega=\qty{0,1,\cdots, 2N}$ be a finite set and let $\mathbb{R}^{\Omega}$ denote the algebra of all functions $f:\Omega\to \mathbb{R}.$
\begin{proposition}\label{IsomoPropRST}
The algebra $\mathcal{A}$ is isomorphic to the algebra $\mathbb{R}^{\Omega}.$ 
\end{proposition}
\begin{proof}
Define a function $\mu:\mathbb{R}^{\Omega}\to \mathcal{A}$ as follows: For every $f\in \mathbb{R}^{\Omega},$
\begin{equation}\label{muEqPBWRST}
\mu(f)(x)=f(\omega) \ \text{ if }x\in I_{\omega},\ \ \omega=0,1,\cdots, 2N.
\end{equation}
We need to prove that $\mu$ is an algebra isomorphism.\par 
Let $f,g\in \mathbb{R}^{\Omega}$ and let $a,b\in \mathbb{R}.$ Then we have the following.
\begin{itemize}
\item If $x\in \mathbb{R},$ then $x\in I_{\omega}$ for some $\omega\in \qty{0,1,\cdots,2N}.$ Therefore
\begin{align*}
\mu(a f+b g)(x)&=(a f+b g)(\omega)\\
&=a f(\omega) +b g(\omega)\\
&=a \mu (f)(x)+b\mu (g)(x)\\
&=[a \mu(f)+b\mu (g)](x)
\end{align*}
That is, $\mu$ is $\mathbb{R}-$linear.
\item \begin{align*}
\mu(fg)(x)&=(fg)(\omega)\\
&=f(\omega)g(\omega)\\
&=\mu (f)(x)\mu (g)(x)\\
&=[\mu(f)\mu(g)](x)
\end{align*}
Therefore, $\mu$ is multiplicative and hence an algebra homomorphism.
\item If for all $x\in \mathbb{R},\ \mu(f)(x)=\mu (g)(x),$ then $f(\omega)=g(\omega)$ for all $\omega\in\qty{0,1,\cdots,2N}.$ That is, $f=g$ and hence $\mu$ is injective.
\item Finally, let $h\in \mathcal{A}$. Then for every $x\in \mathbb{R}$ such that  $x\in I_{\omega},\ h(x)=c_{\omega}$ for some $c_{\omega}\in \mathbb{R}.$  Define $f\in \mathbb{R}^{\Omega}$ by $f(\omega)=c_{\omega},\ \omega=0,1,\cdots,2N.$ If $y\in \mathbb{R}$ such that $y\in I_{\theta}$ for some $\theta\in \qty{0,1,\cdots,2N},$ then 
$$\mu(f)(y)=f(\theta)=c_{\theta}=h(y).$$
Since $y$ is arbitrary, we conclude that $\mu$ is onto.
\end{itemize}
Therefore $\mu$ is an isomorphism.
\end{proof}
Now, let $\sigma:\mathbb{R}\to \mathbb{R}$ be a bijection such that $\mathcal{A}$ is invariant under $\sigma$ (and $\sigma^{-1}$). In \cite[Lemma 1]{AlexABTPBWRST}, it was proved that such a $\sigma$ is a permutation of the partition intervals $I_{\omega},\ \omega=0,1,\cdots,2N.$ Let $\tau:\Omega\to \Omega$ be a bijection (permutation) such that $\tau (\omega)=\theta$ if and only if $\sigma(I_{\omega})=I_{\theta}.$ Suppose $\tilde{\sigma}:\mathcal{A}\to \mathcal{A}$ is the automorphism induced by $\sigma$ and $\tilde{\tau}:\mathbb{R}^{\Omega}\to \mathbb{R}^{\Omega}$ is the automorphism induced by $\tau,$ that is, for every $h\in \mathcal{A} $ and every $f\in \mathbb{R}^{\Omega},$ 
\begin{equation}\label{AutoDefnsPBWRST}
\tilde{\sigma}(h)=h\circ \sigma^{-1} \text{ and }\ \tilde{\tau}(f)=f\circ \tau^{-1}.
\end{equation}
The automorphisms $\tilde{\sigma}$ and $\tilde{\tau}$ satisfy the following intertwining relation.
\begin{proposition}
Let $\sigma:\mathbb{R}\to \mathbb{R}$ be a bijection such that $\mathcal{A}$ is invariant under $\sigma$ (and $\sigma^{-1}$) and let $\tau:\Omega\to \Omega$ be a bijection (permutation) such that $\tau (\omega)=\theta$ if and only if $\sigma(I_{\omega})=I_{\theta}.$ Suppose $\tilde{\sigma}:\mathcal{A}\to \mathcal{A}$ is the automorphism induced by $\sigma$ and $\tilde{\tau}:\mathbb{R}^{\Omega}\to \mathbb{R}^{\Omega}$ is the automorphism induced by $\tau.$ Then
\begin{equation}\label{intertwin1PBWRST}
\tilde{\sigma}\circ \mu =\mu \circ \tilde{\tau},
\end{equation}
where $\mu$ is given by \eqref{muEqPBWRST}. Moreover, for every $n\in \mathbb{Z}$,
\begin{equation}\label{intertwin2PBWRST}
\tilde{\sigma}^n\circ \mu =\mu \circ \tilde{\tau}^n.
\end{equation}
\end{proposition}
\begin{proof}
Let $f\in \mathbb{R}^{\Omega}$ and $x\in \mathbb{R}.$ Suppose $x\in I_{\omega}$ and that $\sigma^{-1}(I_{\theta})=I_{\omega}$ for some $\theta\in \qty{0,1,\cdots,2N.}$ Then,
\begin{align*}
\tilde{\sigma}\circ \mu (f)(x)&=\tilde{\sigma}(\mu(f))(x)\\
&=\mu(f)\qty(\sigma^{-1}(x))\\
&=f\qty(\tau^{-1}(\omega))\\
&=\tilde{\tau}(f)(\omega)\\
&=\mu \circ \tilde{\tau}(f)(x),
\end{align*}
which proves \eqref{intertwin1PBWRST}. The relation \eqref{intertwin2PBWRST} follows by induction.
\end{proof}
In the next Lemma we prove an equivalence between $Sep_{\mathcal{A}}^n(\mathbb{R})$ and $Sep^n(\Omega)$, two sets which will be important in the description of the centralizer of the coefficient algebra in the skew PBW extension. First we give the definitions.
\begin{definition}
For every $n\in \mathbb{Z}$ set,
\begin{equation}
Sep_{\mathcal{A}}^n(\mathbb{R}):=\qty{x\in \mathbb{R}\ :\ (\exists \ h\in \mathcal{A}),\ h(x)\neq \tilde{\sigma}^n(h)(x)},
\end{equation}
\begin{equation}
Sep_{\mathbb{R}^{\Omega}}^n(\Omega):=\qty{\omega\in \Omega\ :\ \qty(\exists \ f\in \mathbb{R}^{\Omega}),\ f(\omega)\neq \tilde{\tau}^n(f)(\omega)},
\end{equation}
and 
\begin{equation}
Sep^n(\Omega):=\qty{\omega\in \Omega\ :\ \tau^n(\omega)\neq \omega}.
\end{equation}
\end{definition}
We have the following.
\begin{lemma}\label{SeplemPBWRST}
Let $x\in I_{\omega}\subset \mathbb{R}.$ Then $x\in Sep_{\mathcal{A}}^n(\mathbb{R})$ if and only if $\omega\in Sep^n(\Omega).$
\end{lemma}
\begin{proof}
Since the algebra $\mathbb{R}^{\Omega}$ separates points, then $Sep_{\mathbb{R}^{\Omega}}^n(\Omega)=Sep^n(\Omega)$ for every $n\in \mathbb{Z}.$ Therefore, it suffices to prove that $x\in I_{\omega}\subset \mathbb{R}$ belongs to $Sep_{\mathcal{A}}^n(\mathbb{R})$ if and only if $\omega\in Sep_{\mathbb{R}^{\Omega}}^n(\Omega).$ To this end, we have the following.\par 
Suppose $\omega\in Sep_{\mathbb{R}^{\Omega}}^n(\Omega).$ Then there exists $f\in \mathbb{R}^{\Omega}$ such that $\tilde{\tau}^n(f)(\omega)\neq f(\omega).$ Since $\mu$ is injective, then $\tilde{\tau}^n(f)(\omega)\neq f(\omega)$ implies that $$\mu \circ \tilde{\tau}^n(f)(x)\neq \mu (f)(x)\ \ \forall \ x\in I_{\omega}.$$
But from \eqref{intertwin2PBWRST}, $\mu \circ \tilde{\tau}^n=\tilde{\sigma}^n\circ \mu.$ Therefore,
$$\tilde{\sigma}^n\qty(\mu (f))(x)\neq \mu(f)(x).$$
That is, $x\in Sep_{\mathcal{A}}^n(\mathbb{R}).$\par 
Conversely, suppose $x\in Sep_{\mathcal{A}}^n(\mathbb{R}).$ Then there exists $h\in \mathcal{A}$ such that $\tilde{\sigma}^n(h)(x)\neq h(x).$ Using injectivity of $\mu^{-1},$ we get 
$$\qty(\mu^{-1}\circ\tilde{\sigma}^n)(h)(\omega)\neq \mu^{-1}(h)(\omega).$$
Again, using \eqref{intertwin2PBWRST}, we get that $\mu^{-1}\circ \tilde{\sigma}^n=\tilde{\tau}^n\circ \mu^{-1}.$ Therefore
$$
\tilde{\tau}^n\qty(\mu^{-1}(h))(\omega)\neq \qty(\mu^{-1}(h))(\omega),
$$
and hence $\omega\in Sep_{\mathbb{R}^{\Omega}}^n(\Omega).$
\end{proof}
From Proposition \ref{IsomoPropRST} and Lemma \ref{SeplemPBWRST} above, it follows that we can consider the algebra of functions on a finite set. Indeed in the following section we consider the skew PBW extension of the algebra $\mathbb{R}^{\Omega}$ of functions on a finite set $\Omega$ and then deduce the corresponding results in the case of the skew PBW extension of the algebra of piecewise constant functions on $\mathbb{R}$ via the isomorphism $\mu$.
\subsection{Algebra of functions on a finite set}
Let $\Omega=\{0,1,\cdots,2N\}$ be a finite set and let $\mathbb{R}^{\Omega}=\qty{f:\Omega\to \mathbb{R}}$ denote the algebra of real-valued functions on $\Omega$ with respect to the usual pointwise operations. By writing $f_k:=f(k)$, $\mathbb{R}^{\Omega}$ can be identified with $\mathbb{R}^{2N+1}$ where $\mathbb{R}^{2N+1}$ is equipped with the usual operations of pointwise addition, scalar multiplication and multiplication defined by
$$xy=(x_1y_1,x_2y_2,\cdots,x_ny_n)$$ for every $x=(x_1,x_2,\cdots,x_n)$ and $y=(y_1,y_2,\cdots,y_n).$\par 
Now, for $1\leqslant i\leqslant n,$ let $\tau_i :\Omega\to \Omega$ be a bijection such that $\mathbb{R}^{\Omega}$ is invariant under $\tau_i$ and $\tau_i^{-1}$, (that is both $\tau_i$ and $\tau_i^{-1}$ are permutations on $\Omega$). For $1\leqslant i\leqslant n$ let $\tilde{\tau}_i:\mathcal{A}\to \mathcal{A}$ be the automorphism induced by $\tau_i,$ that is 
\begin{equation}\label{TauPBWRST}
\tilde{\tau}_i(f)=f\circ \tau_i^{-1}
\end{equation}
for every $f\in \mathbb{R}^{\Omega}$ and let $\delta_i, 1\leqslant i\leqslant n$ be a $\tilde{\tau}_i-$derivation. Consider the skew-PBW extension $\tilde{\tau}\qty(\mathbb{R}^{\Omega})<x_1,\cdots,x_n>.$\par
The following definition is important in the description of the centralizer of $\mathbb{R}^{\Omega}$ in the skew PBW extension $\tilde{\tau}\qty(\mathbb{R}^{\Omega})<x_1,\cdots,x_n>.$ 
\begin{definition}\label{sepdef}
For $\alpha=(\alpha_1,\alpha_2,\cdots,\alpha_n)\in \mathbb{N}^n$, define
\begin{itemize}
\item[(a)] $Sep^{\alpha}(\Omega):=\qty{\omega\in \Omega\ :\ \tau^{\alpha}(\omega)\neq \omega};$
\item[(b)] $Per^{\alpha}(\Omega):=\qty{\omega\in \Omega\ :\ \tau^{\alpha}(\omega)=\omega}.$
\end{itemize}
\end{definition}
\subsection{Centralizers in skew PBW extensions for function algebras}\label{centPBWfunalg}
In this section we describe the centralizer of $\mathbb{R}^{\Omega}$ in the skew PBW extension $\tilde{\tau}\qty(\mathbb{R}^{\Omega})<x_1,\cdots,x_n>.$ We start by describing the centralizer in the quasi-commutative case and then state a necessary condition for an element to belong to the centralizer of $\mathbb{R}^{\Omega}$ in the general case. We finish by giving the  description of the center of the skew PBW extension in the quasi-commutative case. 
\subsubsection{The centralizer of $\mathbb{R}^{\Omega}$}
\begin{theorem}\label{centthmPBWRST}
Suppose that for $1\leqslant i\leqslant n,\ \delta_i=0.$ Then the centralizer $C\qty(\mathbb{R}^{\Omega}),$ of $\mathbb{R}^{\Omega}$ in the skew $PBW$ extension $\tilde{\tau}(\mathbb{R}^{\Omega})<x_1,\cdots,x_n>$ is given by
$$C\qty(\mathbb{R}^{\Omega})=\qty{\sum_{\alpha}f_{\alpha}x^{\alpha}\ :\ f_{\alpha}=0\text{ on }Sep^{\alpha}(\Omega)}.$$
\end{theorem} 
\begin{proof}
Using the results of Theorem \ref{thm1PBWRST}, an element $f=\sum\limits_{\alpha}f_{\alpha}x^{\alpha}\in \tilde{\tau}\qty(\mathbb{R}^{\Omega})<x_1,\cdots,x_n>$ belongs to the centralizer of $\mathbb{R}^{\Omega}$ if and only if for every $r\in \mathbb{R}^{\Omega},$
$$f_{\alpha}\qty(\tilde{\tau}^{\alpha}(r)-r)=0.$$
Since $\qty(\tilde{\tau}^{\alpha}(r)-r)(y)=0$ for every $y\in Per_{\mathbb{R}^{\Omega}}^{\alpha}(\Omega),$ then $fr=rf$ for every $r\in \mathbb{R}^{\Omega}$ if and only if $f_{\alpha}=0$ on $Sep_{\mathbb{R}^{\Omega}}^{\alpha}(\Omega)$. Since $Sep_{\mathbb{R}^{\Omega}}^{\alpha}(\Omega)=Sep^{\alpha}(\Omega),$ we have,
$$C\qty(\mathbb{R}^{\Omega})=\qty{\sum_{\alpha}f_{\alpha}x^{\alpha}\ :\ f_{\alpha}=0\text{ on }Sep^{\alpha}(\Omega)}.$$
\end{proof}
Now consider the case when $\delta_i\neq 0.$ In \cite{Richter1PBWRST} a necessary condition for an element $\sum\limits_{k=0}^mf_kx^k$ in the Ore extension  $\mathbb{R}^{\Omega}[x,\tilde{\tau},\delta]$ to belong to the centralizer of $\mathbb{R}^{\Omega}$ was stated and the following Theorem was proved.
\begin{theorem}\label{centerOre}
If an element of degree $m,$ $\sum\limits_{k=0}^mf_kx^k\in \mathbb{R}^{\Omega}[x,\tilde{\tau},\delta]$ belongs to the centralizer of $\mathbb{R}^{\Omega}$, then $f_m=0 \text{ on } Sep^m(\Omega).$
\end{theorem}
We aim to extend this theorem to the skew PBW extension  $\tilde{\tau}(\mathbb{R})<x_1,\cdots,x_n>$ of which the Ore extension $\mathbb{R}^{\Omega}[x,\tilde{\tau},\delta]$ is a special case. This extension is given in the following Theorem.
\begin{theorem}\label{centNCPBWRST}
If an element $\sum_{\alpha}f_{\alpha}x^{\alpha}\in \tilde{\tau}(\mathbb{R}^{\Omega})<x_1,\cdots,x_n>$ belongs to the centralizer of $\mathbb{R}^{\Omega},$ then $f_{\alpha}=0$ on $Sep^{\alpha}(\Omega).$
\end{theorem}
\begin{proof}
Again, using Theorem \ref{thm2PBWRST}, we see that if an element $f=\sum_{\alpha}f_{\alpha}x^{\alpha}\in \tilde{\tau}(\mathbb{R}^{\Omega})<x_1,\cdots,x_n>$ belongs to the centralizer of $\mathbb{R}^{\Omega},$ then
\begin{equation}\label{centNCeqPBWRST}
\qty(r-\tilde{\tau}^r(r))f_{\alpha}=0 \ \  \forall \ \alpha \in \mathbb{N}^n.
\end{equation}
Equation \eqref{centNCeqPBWRST} holds on $Per_{\mathbb{R}^{\Omega}}^{\alpha}(\Omega)$ and holds on $Sep_{\mathbb{R}^{\Omega}}^{\alpha}(\Omega)$ if $f_{\alpha}=0.$ The conclusion follows from the fact that $Sep_{\mathbb{R}^{\Omega}}^{\alpha}(\Omega)=Sep^{\alpha}(\Omega)$ for all $\alpha\in \mathbb{N}^n.$
\end{proof}
\subsubsection{Center in the quasi-commutative case}
In this section we give the description of the center of the skew PBW extension $\tilde{\tau}\qty(\mathbb{R}^{\Omega})<x_1,\cdots,x_n>$ in the quasi-commutative case, Definition \ref{quasicomm}.  We start with a result which will be important in the description of the center.
 \begin{lem}
 Let $\tilde{\tau}\qty(\mathbb{R}^{\Omega})<x_1,\cdots,x_n>$ be a quasi-commutative PBW extension. Then for every $1\leqslant i,j\leqslant n$ and every $m\in \mathbb{N},$
 \begin{itemize}
 \item[(a)] $x_jx_i^m=\qty(\prod\limits_{k=0}^{m-1}\tilde{\tau}_i^k(c_{ij}))x_i^mx_j.$
 \item[(b)] $x_j^mx_i=\qty(\prod\limits_{k=0}^{m-1}\tilde{\tau}_j^k(c_{ij}))x_ix_j^m.$
 \end{itemize}
 \end{lem}
 \begin{proof}
 \begin{itemize}
 \item[(a)]The case $m=1$ corresponds to condition (c') in Definition \ref{quasicomm}
  and for $m=2$ we have
 \begin{align*}
 x_jx_i^2&=\qty(x_jx_i)x_i\\
 &=\qty(c_{ij}x_ix_j)x_i\\
 &=c_{ij}x_i(x_jx_i)\\
 &=c_{ij}x_i\qty(c_{ij})x_ix_j\\
 &=c_{ij}\tilde{\tau}_i(c_{ij})x_i^2x_j.
 \end{align*}
 Suppose the formula holds for all positive integers up to and including $m.$ Then 
 \begin{align*}
 x_jx_i^{m+1}&=\qty(x_jx_i^m)x_i\\
 &=\qty(\qty(\prod\limits_{k=0}^{m-1}\tilde{\tau}_i^k(c_{ij}))x_i^mx_j)x_i\\
 &=\qty(\prod\limits_{k=0}^{m-1}\tilde{\tau}_i^k(c_{ij}))\qty(x_i^mc_{ij})x_ix_j\\
 &=\qty(\prod\limits_{k=0}^{m-1}\tilde{\tau}_i^k(c_{ij}))\tilde{\tau}_i^mx_i^{m+1}x_j\\
 &=\qty(\prod\limits_{k=0}^{m}\tilde{\tau}_i^k(c_{ij}))x_i^{m+1}x_j.
 \end{align*}
 A similar proof can be done for part (b).
 \end{itemize}
 \end{proof}
 Using these formulas we can derive necessary and sufficient conditions for an element to belong to the center. We state these conditions in the following Theorem.
 \begin{theorem}
 An element $f=\sum\limits_{\alpha}f_{\alpha}x^{\alpha}$ belongs to the center of the quasi-commutative skew PBW extension $\tilde{\tau}\qty(\mathbb{R}^{\Omega})<x_1,\cdots,x_n>$ if and only if $f_{\alpha}=0$ on $Sep^{\alpha}(\Omega)$ and  for every $1\leqslant i\leqslant n$
 \begin{align*}
 \tilde{\tau}_i(f_{\alpha})\prod_{j=1}^{i-1}\qty(\prod_{k_j=0}^{\alpha_j-1}\tilde{\tau}_1^{\alpha_1}\tilde{\tau}_2^{\alpha_2}\cdots\tilde{\tau}_{j-2}^{\alpha_{j-2}}\tilde{\tau}_j^{k_j}(c_{j,i}))=\\
   f_{\alpha}\prod_{j=1}^{i+2}\qty(\prod_{k_{n-j+1}=0}^{\alpha_{n-j+1}-1}\tilde{\tau}_1^{\alpha_1}\cdots\tilde{\tau}_i^{\alpha_1}\tilde{\tau}_{n-i+1}^{\alpha_{n-i+1}}\cdots\tilde{\tau}_{n-j}^{\alpha_{n-j}}\tilde{\tau}_{n-j+1}^{k_{n-j+1}}(c_{i,n-j+1}))
\end{align*}
\end{theorem}  
\begin{proof}
An element $f=\sum\limits_{\alpha}f_{\alpha}x^{\alpha}$ belongs to the center of the quasi-commutative skew PBW extension $\tilde{\tau}\qty(\mathbb{R}^{\Omega})<x_1,\cdots,x_n>$ if and only if $f\in C\qty(\mathbb{R}^{\Omega})$ and, for every $1\leqslant i\leqslant n,\ x_if=fx_i.$ So we compute.
\begin{align*}
x_if&=\sum_{\alpha}x_if_{\alpha}x_1^{\alpha_1}\cdots x_i^{\alpha_i}\cdots x_n^{\alpha_n}\\
&=\sum_{\alpha}\tau_i(f_{\alpha})x_ix_1^{\alpha_1}\cdots x_i^{\alpha_i}\cdots x_n^{\alpha_n}\\
&=\sum_{\alpha}\tau_i(f_{\alpha})\qty(\prod_{k_1=0}^{\alpha_1-1}\tilde{\tau}_1^{k_1}(c_{1,i}))x_1^{\alpha_1}x_ix_2^{\alpha_2}\cdots x_i^{\alpha_i}\cdots x_n^{\alpha_n}\\
&=\sum_{\alpha}\tilde{\tau}_i(f_{\alpha})\qty(\prod_{k_1=0}^{\alpha_1-1}\tilde{\tau}_1^{k_1}(c_{1,i}))\qty(\prod_{k_2=0}^{\alpha_2-1}\tilde{\tau}_1^{\alpha_1}\tilde{\tau}_2^{k_2}(c_{2,i}))x_1^{\alpha_1}x_2^{\alpha_2}x_i\cdots x_i^{\alpha_i}\cdots x_n^{\alpha_n}\\
&\qquad \vdots\\
&=\sum_{\alpha}\tilde{\tau}_i(f_{\alpha})\qty(\prod_{k_1=0}^{\alpha_1-1}\tilde{\tau}_1^{k_1}(c_{1,i}))\qty(\prod_{k_2=0}^{\alpha_2-1}\tilde{\tau}_1^{\alpha_1}\tilde{\tau}_2^{k_2}(c_{2,i}))\cdots \qty(\prod_{k_{i-1}}^{\alpha_{i-1}-1}\tilde{\tau}_1^{\alpha_1}\cdots\tilde{\tau}_{i-2}^{\alpha_{i-2}}\tilde{\tau}_{i-1}^{k_{i-1}}(c_{i-1,i}))\\
&\qquad x_1^{\alpha_1}\cdots x_i^{\alpha_i+1}\cdots x_n^{\alpha_n}\\
&=\sum_{\alpha}\tilde{\tau}_i(f_{\alpha})\prod_{j=1}^{i-1}\qty(\prod_{k_j=0}^{\alpha_j-1}\tilde{\tau}_1^{\alpha_1}\tilde{\tau}_2^{\alpha_2}\cdots\tilde{\tau}_{j-2}^{\alpha_{j-2}}\tilde{\tau}_j^{k_j}(c_{j,i}))x_1^{\alpha_1}\cdots x_i^{\alpha_i+1}\cdots x_n^{\alpha_n}.
\end{align*} 
On the other hand,
\begin{align*}
fx_i&=\qty(\sum_{\alpha}f_{\alpha}x_1^{\alpha_1}\cdots x_i^{\alpha^{\alpha_i}}\cdots x_n^{\alpha_n})\\
&=\sum_{\alpha}f_{\alpha}x_1^{\alpha_1}\cdots x_i^{\alpha^{\alpha_i}}\cdots x_{n-1}^{\alpha_{n-1}}\qty(\prod_{k_n=0}^{\alpha_n-1}\tilde{\tau}_n^{k_n}(c_{i,n}) )x_i x_n^{\alpha_n}\\
&=\sum_{\alpha}f_{\alpha}x_1^{\alpha_1}\cdots x_i^{\alpha^{\alpha_i}}\cdots x_{n-2}^{\alpha_{n-2}}\qty(\prod_{k_n=0}^{\alpha_n-1}\tilde{\tau}_{n-1}^{\alpha_{n-1}}\tilde{\tau}_n^{k_n}(c_{i,n}) )\qty(\prod_{k_{n-1}=0}^{\alpha_{n-1}-1}\tilde{\tau}_{n-1}^{k_{n-1}}(c_{i,n-1}))x_i x_{n-1}^{\alpha_{n-1}}x_n^{\alpha_n}\\
&\qquad \vdots\\
&=\sum_{\alpha}f_{\alpha}x_1^{\alpha_1}\cdots x_i^{\alpha_i}\qty(\prod_{k_n=0}^{\alpha_n-1}\tilde{\tau}_{n-i+1}^{\alpha_{n-i+1}}\tilde{\tau}_{n-i+2}^{\alpha_{n-i+2}}\cdots \tilde{\tau}_{n-1}^{\alpha_{n-1}}\tilde{\tau}_n^{k_n}(c_{i,n}))\\ &\qty(\prod_{k_{n-1}=0}^{\alpha_{n-1}-1}\tilde{\tau}_{n-i+1}^{\alpha_{n-i+1}}\cdots \tilde{\tau}_{n-2}^{\alpha_{n-2}}\tilde{\tau}_{n-1}^{\alpha_{n-1}}(c_{i,n-1}))
\qty(\prod_{k_{n-i-1}=0}^{\alpha_{n-i-1}-1}\tilde{\tau}_{n-i-1}^{k_{n-i-1}}(c_{n-i-1,i}))x_ix_{n-i-1}^{\alpha_{n-i-1}}\cdots x_n^{\alpha_n}\\
&=\sum_{\alpha}f_{\alpha}\prod_{j=1}^{i+2}\qty(\prod_{k_{n-j+1}=0}^{\alpha_{n-j+1}-1}\tilde{\tau}_1^{\alpha_1}\cdots\tilde{\tau}_i^{\alpha_1}\tilde{\tau}_{n-i+1}^{\alpha_{n-i+1}}\cdots\tilde{\tau}_{n-j}^{\alpha_{n-j}}\tilde{\tau}_{n-j+1}^{k_{n-j+1}}(c_{i,n-j+1})) x_1^{\alpha_1}\cdots x_i^{\alpha_i+1}\cdots x_n^{\alpha_n}
\end{align*}
Comparing coefficients of $ x_1^{\alpha_1} \cdots x_{i}^{\alpha_{i+1}}\cdots x_n^{\alpha_n}$ completes the proof of the theorem.
\end{proof}
\subsubsection{Some examples}
In the special case when when $n=2$ we have the following. \par 
 Let $A=\tilde{\tau}(\mathbb{R}^{\Omega})<x_1,x_2>$ and suppose an element $f=\sum\limits_{\alpha_1,\alpha_2}f_{\alpha_1,\alpha_2}x_1^{\alpha_1}x_2^{\alpha_2}\in Z(A).$ Then $x_1f=fx_1$ and $x_2f=fx_2.$ Now
 \begin{align*}
 x_1f&=x_1\qty(\sum\limits_{\alpha_1,\alpha_2}f_{\alpha_1,\alpha_2}x_1^{\alpha_1}x_2^{\alpha_2})\\
 &=\sum\limits_{\alpha_1,\alpha_2}\tilde{\tau}_1(f_{\alpha_1,\alpha_2})x_1^{\alpha_1+1}x_2^{\alpha_2},
 \end{align*}
 and 
 \begin{align*}
 fx_1&=\qty(\sum\limits_{\alpha_1,\alpha_2}f_{\alpha_1,\alpha_2}x_1^{\alpha_1}x_2^{\alpha_2})x_1\\
 &=\sum\limits_{\alpha_1,\alpha_2}f_{\alpha_1,\alpha_2}x_1^{\alpha_1}\qty(x_2^{\alpha_2}x_1)\\
 &=\sum\limits_{\alpha_1,\alpha_2}f_{\alpha_1,\alpha_2}x_1^{\alpha_1}\qty(\prod_{k=0}^{\alpha_2-1}\tilde{\tau}_2^k(c_{12}))x_1x_2^{\alpha_2}\\
 &=\sum\limits_{\alpha_1,\alpha_2}f_{\alpha_1,\alpha_2}\qty(\prod_{k=0}^{\alpha_2-1}\tilde{\tau_1}^{\alpha_1}\tilde{\tau_2}^k(c_{12}))x_1^{\alpha_1+1}x_2^{\alpha_2}.
 \end{align*}
 Therefore $x_1f=fx_1$ if and only if 
 $$\tilde{\tau}_1\qty(f_{\alpha_1,\alpha_2})=f_{\alpha_1,\alpha_2}\qty(\prod_{k=0}^{\alpha_2-1}\tilde{\tau_1}^{\alpha_1}\tilde{\tau_2}^k(c_{12})).$$
 On the other hand,
 \begin{align*}
 fx_2&=\qty(\sum\limits_{\alpha_1,\alpha_2}f_{\alpha_1,\alpha_2}x_1^{\alpha_1}x_2^{\alpha_2})x_2\\
 &=\sum\limits_{\alpha_1,\alpha_2}f_{\alpha_1,\alpha_2}x_1^{\alpha_1}x_2^{\alpha_2+1},
 \end{align*}
 and 
 \begin{align*}
 x_2f&=x_2\qty(\sum\limits_{\alpha_1,\alpha_2}f_{\alpha_1,\alpha_2}x_1^{\alpha_1}x_2^{\alpha_2})\\
 &=\sum\limits_{\alpha_1,\alpha_2}\tilde{\tau}_2\qty(f_{\alpha_1,\alpha_2})x_2x_1^{\alpha_1}x_2^{\alpha_2}\\
 &=\sum\limits_{\alpha_1,\alpha_2}\tilde{\tau}_2\qty(f_{\alpha_1,\alpha_2})\qty(\prod_{k=1}^{\alpha_1-1}\tilde{\tau}_1^k(c_{12}))x_1^{\alpha_1}x_2^{\alpha_2+1}.
 \end{align*}
 Therefore $x_2f=fx_2$ if and only if 
 $$f_{\alpha_1,\alpha_2}=\tilde{\tau}_2\qty(f_{\alpha_1,\alpha_2})\qty(\prod_{k=1}^{\alpha_1-1}\tilde{\tau}_1^k(c_{12})).$$
 We conclude that $f\in Z(A)$ if and only if 
 $$
\tilde{\tau}_1(f_{\alpha_1,\alpha_2})= \tilde{\tau}_2\qty(f_{\alpha_1,\alpha_2})\qty(\prod_{k=1}^{\alpha_1-1}\tilde{\tau}_1^k(c_{12}))\qty(\prod_{k=0}^{\alpha_2-1}\tilde{\tau_1}^{\alpha_1}\tilde{\tau_2}^k(c_{12})).
 $$
 In the next example we give an explicit description of the centralizer of $\mathbb{R}^{\Omega}$ and the center for a particular quasi-commutative skew PBW-extension $\tilde{\tau}(\mathbb{R}^{\Omega})<x_1,x_2>$. Recall that for $m=2$ the algebra $\mathbb{R}^{\Omega}$ is isomorphic to $\mathbb{R}^2.$
\begin{exam}
 \end{exam}
 Consider the quasi-commutative skew PBW extension $A=\tilde{\tau}\qty(\mathbb{R}^{\Omega})<x_1,x_2>$  with the following conditions.
 \begin{itemize}
 \item The automorphisms $\tilde{\tau}_1,\tilde{\tau}_2:\mathbb{R}^2\to \mathbb{R}^2$ are defined as follows:\par 
 $\tilde{\tau}_1=id,\ \tilde{\tau}_2(e_1)=e_2$ and $\tilde{\tau}_2(e_2)=e_1,$ where
 $e_1,e_2$ are the standard basis vectors in $\mathbb{R}^2.$ 
 \item $x_2x_1=(1,2)x_1x_2\ \qty(\Leftrightarrow\ x_1x_2=\qty(1,\frac{1}{2})x_2x_1)$
 \end{itemize}
 From Theorem \ref{centthmPBWRST}, the centralizer of $\mathbb{R}^{\Omega}$ in the skew PBW extension $\tilde{\tau}(\mathbb{R}^{\Omega})<x_1,x_2>$ is given by 
 $$
C\qty(\mathbb{R}^{\Omega})=\qty{\sum_{\alpha}f_{\alpha}x^{\alpha}\ :\ f_{\alpha}=0\text{ on }Sep^{\alpha}(\Omega)}. 
 $$
 In this case 
 $$
Sep^{\alpha}(\Omega)=Sep^{\alpha_1,\alpha_2}(\Omega)=Sep^{\alpha_2}=\begin{cases} \Omega & \text{ if }\alpha_2 \text{ is  odd}\\
\emptyset & \text{ if } \alpha_2 \text{ is even}.  
\end{cases} 
 $$
 Therefore
 $$
C\qty(\mathbb{R}^{\Omega})=\qty{\sum_{j,k}f_{j,2k}x_1^jx_2^{2k}}. 
 $$
 Now let us consider the center. \par 
 Suppose an element $f=\sum\limits_{\alpha}f_{\alpha}x^{\alpha}\in Z(A).$ Then $f\in C\qty(\mathbb{R}^{\Omega}),$ and $x_if=fx_i$ for $i=1,2.$ Since $f\in C\qty(\mathbb{R}^{\Omega}),$ then $f=\sum\limits_{j,k}f_{j,2k}x_1^jx_2^{2k}.$\par  Now
 \begin{align*}
 x_1f&=x_1\qty(\sum\limits_{j,k}f_{j,2k}x_1^jx_2^{2k})\\
 &=\sum\limits_{j,k}\tilde{\tau}_1(f_{j,2k})x_1^{j+1}x_2^{2k}\\
 &=\sum\limits_{j,k}f_{j,2k}x_1^{j+1}x_2^{2k}
 \end{align*}
 and 
 \begin{align*}
 fx_1&=\qty(\sum\limits_{j,k}f_{j,2k}x_1^jx_2^{2k})x_1\\
 &=\sum\limits_{j,k}f_{j,2k}x_1^j\qty(x_2^{2k}x_1)\\
 &=\sum\limits_{j,k}f_{j,2k}x_1^j\qty(\prod_{l=0}^{2k-1}\tilde{\tau}_2^l(1,2))x_1x_2^{2k}\\
 &=\sum\limits_{j,k}f_{j,2k}\tilde{\tau}_1^j\qty(2^k,2^k)x_1^{j+1}x_2^{2k}\\
 &=\sum\limits_{j,k}f_{j,2k}\qty(2^k,2^k)x_1^{j+1}x_2^{2k}.
 \end{align*}
 Therefore $x_1f=fx_1$ if and only if 
 $$f_{j,2k}=f_{j,2k}\qty(2^k,2^k)$$
 from which we obtain that either $f_{j,2k}=0$ for all $j,k$ or $k=0.$
\par 
Also $fx_2=x_2f$ and since $k=0,$ we get $f=\sum\limits_{j}f_jx_1^j.$ Therefore
$$fx_2=\qty(\sum\limits_{j}f_jx_1^j)x_2=\sum\limits_{j}f_jx_1^jx_2,$$ 
on the other hand
\begin{align*}
x_2f&=x_2\qty(\sum\limits_{j}f_jx_1^j)\\
&=\sum\limits_{j}\tilde{\tau}_2(f_j)x_2x_1^j\\
&=\sum\limits_{j}\tilde{\tau}_2(f_j)\qty(\prod_{l=0}^{j-1}\tilde{\tau}_1^l(1,2))x_1^jx_2\\
&=\sum\limits_{j}\tilde{\tau}_2(f_j)(1,2^j)x_1^jx_2
\end{align*}
from which we obtain that $x_2f=fx_2$ if and only if 
$$\tilde{\tau}_2(f)(1,2^j)=f_j.$$
If we suppose $f_j=(a,b)$ then we obtain that $x_2f=fx_2$ if and only if 
$$(b,a)(1,2^i)=(a,b)$$
that is, $a=b$ and $2^j=1$ ($\Rightarrow j=0$ for all $j$). Therefore
$$Z\qty(A)=\qty{\sum_{j}f_jx_1^j\ :\ f_j=k(1,1)\text{ for some } k\in \mathbb{R}}.$$
In following example, we investigate what happens to the center $Z\qty(\tilde{\tau}(\mathbb{R}^{\Omega})<x_1,x_2>)$ if we make a choice of constants  $c_{12}=(c_1,c_2)$ for arbitrary $c_1,c_2\in \mathbb{R}.$
\begin{exam}
 \end{exam}
 Consider the quasi-commutative skew PBW extension $A=\tilde{\tau}\qty(\mathbb{R}^{\Omega})<x_1,x_2>$ with the following conditions.
 \begin{itemize}
 \item The automorphisms $\tilde{\tau}_1,\tilde{\tau}_2:\mathcal{A}\to \mathcal{A}$ are defined as follows:\par 
 $\tilde{\tau}_1=id,\ \tilde{\tau}_2(e_1)=e_2$ and $\tilde{\tau}_2(e_2)=e_1,$ where
 $e_1,e_2$ are the standard basis vectors in $\mathbb{R}^2.$ 
 \item $x_2x_1=(c_1,c_2)x_1x_2\ \qty(\Leftrightarrow\ x_1x_2=\qty(\frac{1}{c_1},\frac{1}{c_2})x_2x_1)$ where $c_1,c_2\in\mathbb{R}$ with $c_1\neq 0\neq c_2.$
 \end{itemize}
 From Theorem \ref{centthmPBWRST}, the centralizer of $\mathbb{R}^{\Omega}$ in the skew PBW extension $\tilde{\tau}\qty(\mathbb{R}^{\Omega})<x_1,x_2>$ is given by 
 $$
C\qty(\mathbb{R}^{\Omega})=\qty{\sum_{\alpha}f_{\alpha}x^{\alpha}\ :\ f_{\alpha}=0\text{ on }Sep^{\alpha}(\Omega)}. 
 $$
 In this case 
 $$
Sep^{\alpha}(\Omega)=Sep^{\alpha_1,\alpha_2}(\Omega)=Sep^{\alpha_2}(\Omega)=\begin{cases} \Omega & \text{ if }\alpha_2 \text{ is  odd}\\
\emptyset & \text{ if } \alpha_2 \text{ is even}.  
\end{cases} 
 $$
 Therefore
 $$
C\qty(\mathbb{R}^{\Omega})=\qty{\sum_{j,k}f_{j,2k}x_1^jx_2^{2k}}. 
 $$
 Now suppose an element $f=\sum\limits_{\alpha}f_{\alpha}x^{\alpha}\in Z(A).$ Then $f\in C\qty(\mathbb{R}^{\Omega}),$ and $x_if=fx_i$ for $i=1,2.$ Since $f\in C\qty(\mathbb{R}^{\Omega}),$ then $f=\sum\limits_{j,k}f_{j,2k}x_1^jx_2^{2k}.$\par  Now
 \begin{align*}
 x_1f&=x_1\qty(\sum\limits_{j,k}f_{j,2k}x_1^jx_2^{2k})\\
 &=\sum\limits_{j,k}\tilde{\tau}_1(f_{j,2k})x_1^{j+1}x_2^{2k}\\
 &=\sum\limits_{j,k}f_{j,2k}x_1^{j+1}x_2^{2k}
 \end{align*}
 and 
 \begin{align*}
 fx_1&=\qty(\sum\limits_{j,k}f_{j,2k}x_1^jx_2^{2k})x_1\\
 &=\sum\limits_{j,k}f_{j,2k}x_1^j\qty(x_2^{2k}x_1)\\
 &=\sum\limits_{j,k}f_{j,2k}x_1^j\qty(\prod_{l=0}^{2k-1}\tilde{\tau}_2^l(c_1,c_2))x_1x_2^{2k}\\
 &=\sum\limits_{j,k}f_{j,2k}\tilde{\tau}_1^j\qty((c_1c_2)^k,(c_1c_2)^k)x_1^{j+1}x_2^{2k}\\
 &=\sum\limits_{j,k}f_{j,2k}\qty((c_1c_2)^k,(c_1c_2)^k)x_1^{j+1}x_2^{2k}.
 \end{align*}
 Therefore $x_1f=fx_1$ if and only if 
 $$f_{j,2k}=f_{j,2k}\qty((c_1c_2)^k,(c_1c_2)^k)$$
 from which we obtain that either $f_{j,2k}=0$ for all $j,k$ or $(c_1c_2)^k=1$ for all $k.$ That is $c_2=\frac{1}{c_1}.$ Therefore we have the following;\par
  If $c_2=c_1^{-1}$,  then from  $f=\sum\limits_{j,k}f_{j,2k}x_1^jx_2^{2k},$ we have,
$$fx_2=\qty(\sum\limits_{j,k}f_{j,2k}x_1^jx_2^{2k})x_2=\sum\limits_{j,k}f_{j,2k}x_1^jx_2^{2k+1}.$$ 
On the other hand
\begin{align*}
x_2f&=x_2\qty(\sum\limits_{j,k}f_{j,2k}x_1^jx_2^{2k})\\
&=\sum\limits_{j,k}\tilde{\tau}_2(f_{j,2k})x_2x_1^jx_2^{2k}\\
&=\sum\limits_{j,k}\tilde{\tau}_2(f_{j,2k})\qty(\prod_{l=0}^{j-1}\tilde{\tau}_1^l(c_1,c_2))x_1^jx_2^{2k+1}\\
&=\sum\limits_{j,k}\tilde{\tau}_2(f_{j,2k})(c_1^j,c_1^{-j})x_1^jx_2^{2k+1},
\end{align*}
from which we obtain that $x_2f=fx_2$ if and only if  $\tilde{\tau}_2(f_{j,2k})(c_1^j,c_1^{-j})=f_{j,2k}.$
If we suppose $f_{j,2k}=(a_{j,2k},b_{j,2k})$ then we obtain that $x_2f=fx_2$ if and only if 
$$(b_{j,2k},a_{j,2k})(c_1^j,c_1^{-j})=(a_{j,2k},b_{j,2k}).$$
That is,  $b_{j,2k}=c_1^{-j}a_{j,2k},$ and hence ,
$$Z\qty(\tilde{\tau}\qty(\mathbb{R}^{\Omega})<x_1,x_2>)=\qty{\sum_{j,k}f_{j,2k}x_1^jx_2^{2k}\ :\ f_{j,2k}=a_{j,2k}(1,c_1^{-j})\text{ for some } a_{j,2k}\in \mathbb{R}}.$$
\subsection{PBW extensions for the algebra of piecewise constant functions}
In section \ref{FunAlgPBWRST}, the algebra $\mathcal{A}$ of piecewise constant functions $h:\mathbb{R}\to \mathbb{R}$ with $N$ fixed jumps at points $t_1, t_2,\cdots,t_N$ was introduced and we proved that this algebra is isomorphic to $\mathbb{R}^{\Omega},$ the algebra of all functions $f:\Omega\to \mathbb{R}$ indexed by $\Omega=\qty{0,1,\cdots,2N}.$ In section \ref{centPBWfunalg}, we gave a description of the centralizer of $\mathbb{R}^{\Omega}$ in the skew PBW extension $\tilde{\tau}\qty(\mathbb{R}^{\Omega})<x_1,x_2,\cdots,x_n>$. Therefore in this section, we  give the centralizer of the coefficient algebra $\mathcal{A}$ in the skew PBW extension $\tilde{\sigma}(\mathcal{A})<x_1,x_2,\cdots,x_n>$ in terms of the isomorphism $\mu:\mathbb{R}^{\Omega}\to \mathcal{A}$ as given in equation \eqref{muEqPBWRST} and $Sep^{\alpha}(\Omega),$ as given in definition \ref{sepdef}. We start with the following definition.
\begin{definition}\label{sepdefPCRST}
For $\alpha=(\alpha_1,\alpha_2,\cdots,\alpha_n)\in \mathbb{N}^n$, define
\begin{itemize}
\item[(a)] $Sep_{\mathcal{A}}^{\alpha}(\mathbb{R}):=\qty{x\in \mathbb{R}\ :\ (\exists\ h\in \mathcal{A})\ :\ \tilde{\sigma}^{\alpha}(h)(x)\neq h(x)};$
\item[(b)] $Per_{\mathcal{A}}^{\alpha}(\mathbb{R}):=\qty{x\in \mathbb{R}\ :\ \tilde{\sigma}^{\alpha}(h)(x)=h(x)}.$
\end{itemize}
\end{definition}
 Using methods similar to the proof of Theorem \ref{centthmPBWRST}, it can be shown that the centralizer of $\mathcal{A}$ in the quasi-commutative skew PBW extension $\tilde{\sigma}(\mathcal{A})<x_1,x_2,\cdots,x_n>$ is given by the following.
 \begin{proposition}
 Suppose that for $1\leqslant i\leqslant n,\ \delta_i=0.$ Then the centralizer $C(\mathcal{A}),$ of $\mathcal{A}$ in the skew $PBW$ extension $\tilde{\sigma}(\mathcal{A})<x_1,\cdots,x_n>$ is given by
$$C(\mathcal{A})=\qty{\sum_{\alpha}h_{\alpha}x^{\alpha}\ :\ h_{\alpha}=0\text{ on }Sep_{\mathcal{A}}^{\alpha}(\mathbb{R})}.$$
 \end{proposition}
 In the next theorem, we give the description of the centralizer of the $\mathcal{A}$ in terms of the isomorphism $\mu$ and $Sep^{\alpha}(X),$ in the quasi-commutative case.
 \begin{theorem}\label{CentralPCPBWRST}
 The centralizer $C(\mathcal{A})$ of $\mathcal{A}$ in the quasi-commutative PBW extension $\tilde{\sigma}(\mathcal{A})<x_1,x_2,\cdots,x_n>$ is given by;
 $$
C(\mathcal{A})=\qty{\sum_{\alpha}h_{\alpha}x^{\alpha}\ :\ \mu^{-1}(h_{\alpha})=0 \text{ on } Sep^{\alpha}(\Omega)}. 
 $$
 where $\mu$ is given by \eqref{muEqPBWRST}.
 \end{theorem}
 \begin{proof}
 Define a map $\gamma:\mathbb{R}\to \Omega$ such that
 $$\gamma(x)=\omega\ \text{ if }x\in I_{\omega},\ \omega\in \qty{0,1,\cdots,2N}.$$
 Then for every $f\in \mathbb{R}^{\Omega}$ and every $x\in I_{\omega}$,
 $$\mu(f)(x)=f(\omega)=(f\circ \gamma) (x),$$ 
 and for every $h\in \mathcal{A},$ $$\mu^{-1}(h)(\omega)=h(x)\ \text{ for all }x\in \gamma^{-1}(\omega),$$
 where $\gamma^{-1}(\omega)$ denotes the pre-image of $\omega.$ Observe that
 $$\qty{\sum_{\alpha}h_{\alpha}x^{\alpha}\ :\ \mu^{-1}(h_{\alpha})=0 \text{ on } Sep^{\alpha}(\Omega)}=\qty{\sum_{\alpha}h_{\alpha}x^{\alpha}\ :\ h_{\alpha}(x)=0\ \forall\ x\in \gamma^{-1}\qty(Sep^{\alpha}(\Omega))}.$$
 Therefore it remains to prove that $\gamma^{-1}\qty(Sep^{\alpha}(\Omega))=Sep_{\mathcal{A}}^{\alpha}(\mathbb{R}).$ To this end we have the following;
 \begin{align*}
 \gamma^{-1}\qty(Sep^{\alpha}(\Omega))&=\gamma^{-1}\qty(\qty{\omega\in \Omega\ :\ \tau^{\alpha}(\omega)\neq \omega})\\
 &=\qty{\gamma^{-1}(\omega)\ :\ \tau^{\alpha}(\omega)\neq \omega}\\
 &=\qty{x\in \mathbb{R}\ :\ (x\in I_{\omega})\ :\ \tau^{\alpha}(\omega)\neq \omega}\\
 &=\qty{x\in \mathbb{R}\ :\ x\in I_{\omega} \ \text{ and }\sigma^{\alpha}(I_{\omega})\cap I_{\omega}=\emptyset}\\
 &=\qty{x\in \mathbb{R}\ :\ (\exists\ h\in \mathcal{A})\ :\ \tilde{\sigma}^{\alpha}(h)(x)\neq h(x)}\\
 &=Sep_{\mathcal{A}}^{\alpha}(\mathbb{R}).
 \end{align*}
 This completes the proof.
 \end{proof}
 Using the same methods in the proof of Theorem \ref{centNCPBWRST} and from Theorem \ref{CentralPCPBWRST} above, we have the following necessary condition in the general case.
 \begin{theorem}
If an element $\sum_{\alpha}f_{\alpha}x^{\alpha}\in \tilde{\tau}(\mathbb{R}^{\Omega})<x_1,\cdots,x_n>$ belongs to the centralizer of $\mathbb{R}^{\Omega},$ then $\mu^{-1}(h_{\alpha})=0$ on $Sep^{\alpha}(\Omega).$
\end{theorem}
\begin{exam}
 \end{exam}
 Consider the quasi-commutative skew PBW extension $A=\tilde{\tau}\qty(\mathbb{R}^{\Omega})<x_1,x_2>$ with the following conditions.
 \begin{itemize}
 \item The automorphisms $\tilde{\tau}_1,\tilde{\tau}_2:\mathcal{A}\to \mathcal{A}$ are defined as follows:\par 
 $\tilde{\tau}_1=id,\ \tilde{\tau}_2(e_1)=e_2$ and $\tilde{\tau}_2(e_2)=e_1,$ where
 $e_1,e_2$ are the standard basis vectors in $\mathbb{R}^2.$ 
 \item $x_2x_1=(c_1,c_2)x_1x_2\ \qty(\Leftrightarrow\ x_1x_2=\qty(\frac{1}{c_1},\frac{1}{c_2})x_2x_1)$ where $c_1,c_2\in\mathbb{R}$ with $c_1\neq 0\neq c_2.$
 \end{itemize}
This corresponds to the algebra $\mathcal{A}$ of piecewise constant functions with one fixed jump point $t_1$ with $\mathbb{R}$ partitioned into intervals  $I_0=(-\infty, t_1),\ I_1=(t_1,\infty)$ and $I_3=t_1.$ Invariance of $\mathcal{A}$ under any bijection $\sigma:\mathbb{R}\to \mathbb{R}$ implies that $\sigma(t_1)=t_1.$\par 
From the definition of the automorphisms $\tilde{\tau}_1,\tilde{\tau}_2$ we see that the corresponding bijections $\sigma_1,\sigma_2:\mathbb{R}\to \mathbb{R}$ be have as follows
\begin{itemize}
\item $\sigma_1(t_1)=\sigma_2(t_1)=t_1.$
\item $\sigma_1(I_0)=I_0$ and hence $\sigma_1(I_1)=I_1.$
\item $\sigma_2(I_0)=I_1$ which implies $\sigma_2(I_1)=I_0.$
\end{itemize}
From Theorem \ref{CentralPCPBWRST}, it follows that for every $\alpha=(\alpha_1,\alpha_2)\in \mathbb{N}^2,$
$$
\gamma^{-1}\qty(Sep^{\alpha}(\Omega))=Sep_{\mathcal{A}}^{\alpha}(\mathbb{R})=\begin{cases}I_0\cup I_1 &\ \text{ if }\alpha_2 \text{ is odd}\\ \emptyset, & \text{ if $\alpha_2$ is even}\end{cases}.
$$
Therefore the centralizer of $\mathcal{A}$ in the skew PBW extension $\tilde{\sigma}(\mathcal{A})<x_1,x_2>$ is given by
$$
C(\mathcal{A})=\qty{\sum_{j,k}h_{j,k}x_1^{j}x_2^k\ :\ h_{j,2k+1}=0 \text{ on } I_0\cup I_1}. 
 $$

\subsubsection*{Acknowledgement}
This research was supported by the Swedish International Development Cooperation Agency (Sida), International Science Programme (ISP) in Mathematical Sciences (IPMS), Eastern Africa Universities Mathematics Programme (EAUMP).  Alex Behakanira Tumwesigye is also grateful to the research environment Mathematics and Applied Mathematics (MAM), Division of Applied Mathematics, M\"alardalen University for providing an excellent and inspiring environment for research education and research.

\end{document}